\documentclass[a4paper,12pt]{amsart}

\usepackage{geometry}
\usepackage{amsmath,amsfonts,amsthm, amssymb,mathtools}
\usepackage{hyperref, color}

\DeclareMathOperator*{\esssup}{ess\,sup}

\usepackage{soul}

\newcommand{\rdd}{\mathbb{R}^{2d}}
\newcommand{\rd}{\mathbb{R}^{d}}

\newcommand{\C}{\mathbb{C}}
\newcommand{\bN}{\mathbb{N}}

\newcommand{\cS}{\mathcal{S}}

\newenvironment{displaytext}
{\begin{equation}
	\begin{tabular}{@{}l@{}}}
{  \end{tabular}
	\end{equation}}

\newcommand{\red}[1]{#1}

\let\originalleft\left
\let\originalright\right
\renewcommand{\left}{\mathopen{}\mathclose\bgroup\originalleft}
\renewcommand{\right}{\aftergroup\egroup\originalright}

\def\lc{\left(}
\def\rc{\right)}
\def\lV{\Big\|}
\def\rV{\Big\|}

\newtheorem{theorem}{Theorem}[section]
\newtheorem{lemma}[theorem]{Lemma}
\newtheorem{proposition}[theorem]{Proposition}
\newtheorem{remark}[theorem]{Remark}

\numberwithin{equation}{section}

\title[On the existence of time-frequency optimizers]{On the existence of optimizers for time-frequency concentration problems}

\author[F. Nicola]{Fabio Nicola}
\address{Dipartimento di Scienze Matematiche, Politecnico di Torino, Corso Duca degli Abruzzi 24, 10129 Torino, Italy}
\email{fabio.nicola@polito.it}

\author[J. L. Romero]{Jos\'e Luis Romero}
\address{Faculty of Mathematics, University of Vienna, Oskar-Morgenstern-Platz 1, A-1090 Vienna, Austria, and, Acoustics Research Institute, Austrian Academy of Sciences, Wohllebengasse 12-14 A-1040, Vienna, Austria}
\email{jose.luis.romero@univie.ac.at}

\author[S. I. Trapasso]{S. Ivan Trapasso}
\address{MaLGa Center - Department of Mathematics, University of Genova, via Dodecaneso 35, 16146 Genova, Italy}
\email{salvatoreivan.trapasso@unige.it}
\date{}

\begin{document}
\begin{abstract}
We consider the problem of the maximum concentration in a fixed \red{measurable} subset $\Omega\subset\rdd$ of the time-frequency space for functions $f\in L^2(\rd)$.  
The notion of concentration can be made mathematically precise by considering the $L^p$-norm on $\Omega$ of some time-frequency distribution of $f$ such as the ambiguity function $A(f)$. We provide a positive answer to an open maximization problem, by showing that for every subset $\Omega\subset\rdd$ of finite measure and every $1\leq p<\infty$, there exists an optimizer for
\[
 \sup\{\|A(f)\|_{L^p(\Omega)}:\ f\in L^2(\rd),\  \|f\|_{L^2}=1
 \}.
 \]
The lack of weak upper semicontinuity and the invariance under time-frequency shifts make the problem challenging. The proof is based on concentration compactness with time-frequency shifts as dislocations, and certain integral bounds and asymptotic decoupling estimates for the ambiguity function. We also discuss the case $p=\infty$ and related optimization problems for the time correlation function, the cross-ambiguity function with a fixed window, and for functions in the modulation spaces $M^q(\rd)$, $0<q<2$, equipped with continuous or discrete-type (quasi-)norms. 
\end{abstract}

\maketitle

\section{Introduction and discussion of the main results}
The notion of concentration of a function $f\in L^2(\rd)$ in a \red{measurable} subset $\Omega\subset\rdd$ of the time-frequency space is central in harmonic analysis and is also at the core signal processing \cite{flandrin_new,mallat_book,grochenig_book, vetterli}. From a mathematical point of view, \red{the study of this issue represents a fascinating
and multifaceted challenge, with a longstanding and distinguished tradition \cite{fuchs,landau1,landaupol,slepian,daubechies}, and it ultimately reduces to one of the different, subtle manifestations of the uncertainty principle \cite{fefferman,folland_up,ricaud_up}.}

A natural family of phase-space concentration measures is given by the $L^p$ norms on $\Omega$ of some \emph{time-frequency distribution} of $f$, such as the short-time Fourier transform (see below), or the \textit{ambiguity function} 
\[
A(f)(x,\omega)=\int_{\rd} f \lc t+\frac{x}{2}\rc \overline{f \lc t-\frac{x}{2}\rc} e^{-2\pi i t\cdot\omega}\, dt,\quad x,\omega\in\rd,
\]
which is a quadratic time-frequency representation popular in engineering and radar applications \cite{cook,woodward,mabohl13}. 

The design of maximally concentrated waveforms is of great theoretical and practical interest, as these provide compact elementary blocks tailored to a given tiling of the time-frequency space, according to a paradigm that dates back to the pioneering work by Gabor \cite{Gab46} at least.
The companion problem of designing pulses with a peaky ambiguity
function is of particular relevance in radar signal analysis \cite[Section 3.4.3]{ricaud_up}, wireless communications \cite{benedetto,charly_wire,rihacz}, and signal recovery.

In spite of the importance of the problem and the extensive numerical experimentation (see e.g. \cite{fei_opt,Ricaud2014}), the existence of an optimizer for the functional $\|A(f)\|_{L^p(\Omega)}$, $1\leq p<\infty$, among the functions $f\in L^2(\rd)$, $\|f\|_{L^2}=1$, is still open. While this fact can seem surprising given the maturity of 
the field of time-frequency analysis, close inspection of the problem soon reveals a number of technical difficulties, including the lack of weak upper semicontinuity of the involved functional and its invariance with respect to a non-compact group of \emph{time-frequency shifts}:
\begin{align}\label{eq_tf}
\pi(z) f(t) = e^{2\pi i t \cdot \omega} f(t-x), \qquad z=(x,\omega) \in \rd\times\rd.
\end{align}
Our main result establishes the existence of optimizers for the $L^p$-norm of the ambiguity function on a domain. 
\begin{theorem}\label{mainthm}
Let $\Omega\subset\rdd$ be a measurable subset of finite, positive measure, and $1\leq p<\infty$. Then the supremum 
\begin{equation}\label{uno}
    \sup_{f\in L^2(\rd)\setminus\{0\}}\frac{\lc\int_\Omega |A(f)(x,\omega)|^p dx d\omega\rc^{1/p}}{\|f\|^2_{L^2}}
\end{equation}
    is attained.
    
    Moreover, for $1<p<\infty$, if $f^{(n)}$ is any maximizing sequence normalized in $L^2(\rd)$, then there exists a subsequence (still denoted by $f^{(n)}$) and $z^{(n)}\in\rdd$ such that $\pi(-z^{(n)})f^{(n)}$ converges in $L^2$ to a maximizer.
\end{theorem}

%To better discuss this result, let us introduce some notation. Define the translation and modulation operators in $L^2(\rd)$ as 
%\[
%T_xf(t)=f(t-x),\quad  M_\omega f(t)=e^{2\pi i t\cdot\omega}f(t),\quad x,\omega\in \rd,\ f\in L^2(\rd),\]
%and the corresponding time-frequency shifts $\pi(x,\omega)=M_\omega T_x$. 

The optimization objective \eqref{uno} is invariant under time-frequency shifts, since $|A(\pi(x,\omega)f)|=|A(f)|$ for $x,\omega\in\rd$, $f\in L^2(\rd)$. The first step towards Theorem \ref{mainthm} is to account for such symmetries. At the outset, our proof is based on a \textit{concentration compactness} strategy \cite{brenir,uhl,struwe,lions2, lions1,tintarev_book}, where the time-frequency shifts $\{\pi(x,\omega)$: $x,\omega\in \rd\}$ serve as dislocation operators. The corresponding profile decompositions of maximizing sequences are then leveraged by means of certain integral estimates for the
ambiguity function from \cite{cordero_nicola_2018} --- \red{expressing continuity at an intermediate level between the ``dislocation topology" and the weak topology} --- and an asymptotic decoupling property in $L^p(\Omega)$ for sums of functions asymptotically separated \textit{in the Fourier domain}. The latter can be aptly regarded as an asymptotic version of a known almost-orthogonality principle, cf. \cite[Lemma 6.1]{tao_vargas}. Our method yields not only the existence of optimizers, but, for $1<p<\infty$, also implies that every normalized maximizing sequence is relatively compact in $L^2(\rd)$, up to time-frequency shifts. This stronger conclusion is consistent with numerical practices that
seek to optimize \eqref{uno} by fixing a time-frequency center of gravity
\cite{Ricaud2014}.

The attainability of \eqref{uno} in the whole time-frequency space ($\Omega=\rdd$) was studied in the celebrated article \cite{lieb_opt}, under the assumption $p\geq 2$ --- which is a necessary restriction in that case --- and with very different techniques, \red{in particular exploiting the explicit expression of the candidate maximizers. Indeed}, the value of \eqref{uno} was exactly calculated and maximizers were characterized as Gaussian functions. For domains $\Omega$ with a special geometry, such as a ball, a similar characterization could be expected. Theorem \ref{mainthm} is a first step in that direction, as it implies that maximizers exist and therefore satisfy a certain variational equation. The analysis of such equation is however challenging and we postpone it to a subsequent contribution (in preparation) --- cf.\ \cite{nicola_tilli} for a related problem. 

We stress that the conclusion of Theorem \ref{mainthm} does not extend to the case $p=\infty$. Instead, we have the following characterization.
\begin{proposition}\label{pro counter 2} Let $\Omega\subset\rdd$ be a measurable subset of finite, positive measure. Then
\begin{equation}\label{eq count 1} 
\sup_{f\in L^2(\rd)\setminus\{0\}}\frac{\|A(f)\|_{L^\infty(\Omega)}}{\|f\|_{L^2}^2}=1
\end{equation}
and the supremum is attained if and only if  $|\Omega \cap B_r|>0$ for every $r>0$, where $B_r=\{z\in\rdd: |z|<r\}$. In this case, every $f\in L^2(\rd)\setminus\{0\}$ is a maximizer.
\end{proposition}

The magnitude of the ambiguity function $|A(f)(x,\omega)|=|\langle f,\pi(x,\omega)f \rangle|$ is a \emph{time-frequency auto-correlation function}. To better appreciate the subtleties involved in its optimization, we show that a result similar to Theorem \ref{mainthm} fails \red{for} time or frequency correlations considered individually.
Indeed, denote the translation and modulation operators on $L^2(\rd)$ by 
\begin{equation}\label{eq_tm}
T_xf(t)=f(t-x),\quad  M_\omega f(t)=e^{2\pi i t\cdot\omega}f(t),\quad x,\omega\in \rd,\ f\in L^2(\rd),
\end{equation}
so that $\pi(x,\omega)=M_\omega T_x$. The following result is in stark contrast with Theorem \ref{mainthm}.
\begin{proposition}\label{pro counterexample}
Let $\Omega\subset\rd$ be a measurable subset of finite, positive measure, and $1\leq p<\infty$. Then
\begin{equation}\label{uno-bis}
    \sup_{f\in L^2(\rd)\setminus\{0\}}\frac{\lc\int_\Omega |\langle f,T_x f\rangle|^p dx\rc^{1/p}}{\|f\|^2_{L^2}}=|\Omega|^{1/p}
\end{equation}
  and the supremum  is not attained. 
\end{proposition}
Of course, a similar negative result holds true for the frequency correlation function $\langle f,M_\omega f\rangle= \langle \hat{f},T_\omega \hat{f}\rangle$. As $\langle f,T_x f\rangle=(f\ast f^\vee)(x)$, with $f^\vee(x)=\overline{f(-x)}$, \red{Proposition \ref{pro counterexample} could be rephrased as an optimization problem for positive definite functions and it is related, at least in spirit}, to the optimization of the constants in Young's inequality; see, e.g., \cite[Chapter 4]{lieb_book}. 

Optimization problems analogous to \eqref{uno} can be considered also for linear time-frequency representations, such as the \emph{short-time Fourier transform} $V_g f(x,\omega)=\langle f, \pi(x,\omega)g\rangle$, where $g\in L^2(\rd)\setminus\{0\}$ is a fixed window function. While the short-time Fourier transform is not intrinsically associated with the function $f$, as it requires the introduction of an additional parameter $g$, it is a popular tool in signal analysis, in part because it is mathematically simpler than the ambiguity function. For example, the existence of a maximizer for $\|V_g f\|_{L^p(\Omega)}$ is much easier to establish than Theorem \ref{mainthm}, because the introduction of the window function $g$ weakens the nonlinearity of the optimization objective, replacing the quadratic term $|A(f)|$ with the so-called \emph{cross-ambiguity} $|A(f,g)|=|V_g f|$. A proof of the existence of optimizers for $\|V_g f\|_{L^p(\Omega)}$ and a technical comparison to Theorem \ref{mainthm} is presented in Section \ref{sec fixed wind}. (The case
$p=2$ is straightforward, as it corresponds to the maximization of the eigenvalues of a so-called \emph{localization operator} \cite{daubechies,wong_book}; finer questions such as optimal domains \red{of prescribed measure} and characterization of extremizers for Gaussian windows are studied in \cite{nicola_tilli}.)

We also point out that different but related optimization problems have been considered in the literature over the years, such as maximizing the {\it integral} on a subset $\Omega\subset\rdd$ of time-frequency distributions in the Cohen class \cite{cohen} (as opposed to their $L^p$-norms); see for instance \cite{flandrin_opt,lieb_ost}. For this kind of optimization, we refer the reader to the comprehensive recent survey \cite{lerner2021integrating}, where deep connections with the spectral theory of pseudo-differential operators are discussed. 

Finally, we stress once again that $|A(f)(x,\omega)|=|\langle f,\pi(x,\omega)f \rangle|$, so that \eqref{mainthm} can be also regarded as an optimization problem for the diagonal matrix coefficients of the Schr\"odinger representation of the reduced Heisenberg group. This point of view encourages us to investigate for which other groups and unitary representations a similar property holds -- an interesting question that appears to be largely open at the time of writing. Indeed, the matrix coefficients encode the properties of the corresponding representation and their study has a well-established tradition \cite{bargmann,kunzestein,cowling,knapp,ehren}, focused on proving refined estimates on the whole group -- as opposite to a subset $\Omega$.     

The article is organized as follows. In Section \ref{sec notation} we provide brief background on time-frequency analysis and concentration compactness. Section \ref{sec proof mainthm} is devoted to the proof of Theorem \ref{mainthm}, whereas in Section \ref{sec proof propositions} we prove Propositions \ref{pro counter 2} and \ref{pro counterexample}. In Section \ref{sec fixed wind} we consider the optimization problem for $\|A(f,g)\|_{L^p(\Omega)}$ and discuss technical differences with respect to Theorem \ref{mainthm}. Finally, Section \ref{sec_var} provides two variants of Theorem \ref{mainthm}. There, we replace $L^2(\rd)$ by the \emph{modulation spaces} $M^q(\rd)$, $0<q<2$, which are (quasi-)Banach spaces defined by imposing certain integrability requirements to the short-time Fourier transform, \red{widely used in time-frequency analysis} \cite{benyimodulation}. More precisely, we incorporate modulation-space norms into the optimization objective \eqref{uno}, and also consider their often preferred discrete counterparts.

\section{Notation and preliminary results}\label{sec notation}

\subsection{General notation} The inner product in $L^2(\rd)$ is denoted by $\langle\cdot,\cdot\rangle$. The space of Schwartz functions in $\rd$ is denoted by $\mathcal{S}(\rd)$, while $\mathcal{S}'(\rd)$ stands for the space of temperate distributions. We write $A\lesssim B$ if $A\leq C B$ for some absolute constant $C>0$, whereas $A\lesssim_k B$ means that such a constant depends on the parameter $k$. The Lebesgue measure of a subset $E$ of $\rd$ (or $\rdd$) is denoted by $|E|$ while $\chi_E$ stands for its characteristic function.   

\subsection{Tools from time-frequency analysis}
We recall some definitions and facts from time-frequen\-cy analysis; see \cite{grochenig_book, folland, flandrin_old} for extensive background. The main objects are the time-frequency shifts \eqref{eq_tf}, which define a \emph{unitary projective representation} $z\mapsto\pi(z)$ of $\rdd$ on $L^2(\rd)$.
In particular, for all $z,z'\in\mathbb{R}^{2d}$,
\begin{align}\label{eq_uno}
\pi(z) \pi (z') &= \sigma(z,z') \pi(z+z'),
\\
\label{eq_dos}
\pi(z)^* &= \overline{\sigma(z,-z)} \pi(-z),
\end{align}
where $\sigma(z,z')$ is an adequate unimodular complex number, called \emph{cocycle}.

\subsubsection*{Function spaces}
%Recall the translation and modulation operators $T_x$, $M_\omega$, $x,\omega\in\rd$, as well as of the time-frequency shifts $\pi(z)$, $z\in\rdd$. The map $\rdd\ni z\mapsto\pi(z)$ defines a \emph{unitary projective representation of} $\rdd$ on $L^2(\rdd)$. 
We now fix a \emph{window function} $g\in \mathcal{S}(\rd)\setminus\{0\}$. Most definitions below depend (albeit non-essentially) on such choice. The short-time Fourier transform (STFT) of a temperate distribution $f\in \mathcal{S}'(\rd)$ is defined by
\begin{equation}\label{eq_stft}
V_gf(z)= \langle f,\pi(z) g\rangle,\quad z\in\rdd.
\end{equation}
By considering the $L^p$ norm of $V_g f$ in $\rdd$ one can naturally measure the time-frequency content of a distribution and introduce corresponding families of function spaces. For example, the \emph{modulation space} $M^\infty(\rd)$ consists of temperate distributions $f\in\mathcal{S}'(\rd)$ such that 
\[
\|f\|_{M^\infty}\coloneqq \sup_{z\in\rdd}|V_g f(z)|=\sup_{z\in\rdd}|\langle f, \pi(z)g\rangle|<\infty.
\]
Different choices of the window produce equivalent norms, and $L^2(\rd)\hookrightarrow M^\infty(\rd)$.  

The Wiener amalgam space $\mathcal{W}(L^2,L^\infty)$ consists of all measurable functions $f:\rd\to \C$ such that 
\[
\|f\|_{\mathcal{W}(L^2,L^\infty)}\coloneqq  \esssup_{y\in\rd}\|f \,T_y g\|_{L^2}<\infty,
\]
where $T_y$ is the translation \eqref{eq_tm}. Again, different windows give rise to equivalent norms and $L^2(\rd)\hookrightarrow \mathcal{W}(L^2,L^\infty)$. Notice that the reverse inclusion holds locally, namely
\begin{equation}\label{eq l2 embed}
    \| f\|_{L^2(K)}\lesssim_K \|f\|_{\mathcal{W}(L^2,L^\infty)},
\end{equation}
for every compact subset $K\subset\rd$. This follows immediately if the window $g$ is chosen so that $g=1$ in a sufficiently large ball, so that $f=f\,T_y g$ on $K$ for sufficiently small $y$.

\subsubsection*{Cross-ambiguity function}  The cross-ambiguity function 
of $f,g\in L^2(\rd)$ is
\[
A(f,g)(x,\omega)=\int_{\rd} f \lc t+\frac{x}{2}\rc \overline{g\lc t-\frac{x}{2}\rc} e^{-2\pi i t\cdot\omega}\, dt= e^{\pi i x\cdot \omega}V_g f(x,\omega), \qquad x,\omega\in\rd.
\]
Hence $A(f,f)=A(f)$ is the ambiguity function of $f$. It is easy to see that $A(f,g)$ is a continuous function in $\rdd$ and vanishes at infinity. Moreover, by \red{the Cauchy-Schwarz inequality},
\[
|A(f,g)|\leq \|f\|_{L^2}\|g\|_{L^2}. 
\]
The following estimate from \cite[Corollary 4.2]{cordero_nicola_2018} will play a crucial role:
\begin{equation}\label{eq est l2minfty}
    \|A(f,g)\|_{\mathcal{W}(L^2,L^\infty)}\lesssim \|f\|_{L^2}\|g\|_{M^\infty},
\end{equation}
where of course the space $\mathcal{W}(L^2,L^\infty)$ is understood in $\rdd$. While we only need \eqref{eq est l2minfty} for $f,g\in L^2(\rd)$, the formula is still valid for $f\in L^2(\rd)$ and $g\in M^\infty(\rd)$. (In that case, $A(f,g)$ is a priori defined only as a temperate distribution, and part of the content of \eqref{eq est l2minfty} is that $A(f,g)$ is in fact locally in $L^2$ when the right-hand side is finite). Thus, as in many other parts of the article, we are concerned with $L^2(\rd)$ equipped with the $M^\infty$ (norm) topology, but not with genuine distributions in $M^\infty(\rd)$.

We finally recall from \cite[Chapter 4]{folland} or \cite[Proposition 175 and Corollary 217]{degosson} the covariance property for the cross-ambiguity function: if $\mathcal{A}$ is a (real $2d\times 2d$) symplectic matrix we have 
\begin{equation}\label{covariance}
A(U_{\mathcal{A}}f,U_{\mathcal{A}}g)(z)=A(f,g)(\mathcal{A}z)\quad z\in\rdd,\ f,g\in L^2(\rd)
\end{equation}
for a suitable unitary operator $U_{\mathcal{A}}$ on $L^2(\rd)$ (called metaplectic operator). 

\subsection{Tools from concentraction compactness}\label{subsec conc comp}
 Concentration compactness is a general paradigm to study optimization problems when compactness arguments fail due to invariance under the action of a non-compact group (see e.g. \cite{lions1,lions2,tintarev_book,tao}). We recall some basic facts specialized to the (projective) representation given by the time-frequency shifts. The main conclusion is that any bounded sequence in $L^2(\rd)$ has a subsequence with a special profile decomposition.
\subsubsection*{Time-frequency shifts as dislocations in $L^2$}
It is easily checked that time-frequency shifts satisfy the following important \emph{dislocation property}:
\begin{displaytext}\label{eq_dis}
\!\!If $f\in L^2(\rd)$ and $z_n\in\C$ is a sequence with  $|z_n|\to+\infty$,\\then $\pi(z_n) f$ converges weakly to $0$ in $L^2(\rd)$.
\end{displaytext}
\!\!\noindent
The dislocation property allows us to apply the theory of concentration compactness, because it implies that the time-frequency shifts $\{\pi(z): z\in\rdd\}$ define a so-called \textit{dislocation set} of unitary operators in $L^2(\rd)$ (\cite[Definition 3.2]{tintarev_book}). Whereas we do not need to recall the general (technical) definition here, we observe that, according to \cite[Proposition 3.1]{tintarev_book}, it is sufficient to check that
\begin{displaytext}\label{eq_dis2}
\!\!If for some sequences $z_n,z'_n\in\rdd$, the sequence of operators $\pi(z'_n)^\ast\pi(z_n)$ \\ does not converge weakly to zero (as bounded operators on $L^2(\rd)$),\\ then it has a strongly convergent subsequence.
\end{displaytext}
To see that property \eqref{eq_dis2} holds, note first that, by \eqref{eq_uno} and \eqref{eq_dos}, $\pi(z'_n)^\ast\pi(z_n)=c(z_n,z'_n)\pi(z_n-z'_n)$, with $|c(z_n,z'_n)|=1$. Hence, if $\pi(z'_n)^\ast\pi(z_n)$, does not converge weakly to zero, by \eqref{eq_dis}, $|z_n-z'_n|$ does not tend to $+\infty$, and therefore has a convergent subsequence. By passing to a further subsequence, the phase factors $c(z_n,z'_n)$ will also converge, and the conclusion  follows from the strong continuity of the representation
$\rdd\ni z\mapsto \pi(z)$.
\subsubsection*{D-weak convergence}
Associated with the set of dislocations $\{\pi(z): z\in\rdd\}$, there is a corresponding notion of \emph{weak dislocation convergence} --- \textit{$D$-weak convergence} for short \cite[Definition 3.1]{tintarev_book}): a sequence $f_n$ in $L^2(\rd)$ $D$-weakly converges to $f\in L^2(\rd)$ if for every $g\in L^2(\rd)$:
\[
\sup_{z\in{\rdd}} |\langle f_n-f,\pi(z) g\rangle|\to 0.
\]
Letting $g$ be a window function for the short-time Fourier transform \eqref{eq_stft} one sees that $D$-weak convergence implies convergence in $M^\infty(\rd)$.

\subsubsection*{Profile decomposition}
The general theory of concentration compactness in Hilbert spaces (see e.g. \cite[Theorem 3.1 and its proof]{tintarev_book} or \cite[Theorem 4.5.3]{tao}) 
now yields the following: Let $f^{(n)}$ be a sequence in $L^2(\rd)$ with $\limsup_{n\to\infty}\|f^{(n)}\|_{L^2}\leq1$, then there exists a subsequence (that we still denote $f^{(n)}$) and \emph{profiles} $f_j\in L^2(\rd)$, $j=1,2,\ldots,$ such that the following \emph{profile decomposition} holds for $k=1,2,\ldots$:
\begin{equation}\label{due}
f^{(n)}=\sum_{j=1}^k \pi\big(z^{(n)}_j\big)f_j+w^{(n)}_k
\end{equation}
for suitable $z^{(n)}_j\in\rdd$, $w^{(n)}_k\in L^2(\rd)$, where
\begin{equation}\label{tre}
|z^{(n)}_j-z^{(n)}_{j'}|\to+\infty\quad\textrm{as } n\to\infty,\ \textrm{if}\ j\not=j' \mbox{ and }\ f_j\not=0,f_{j'}\not=0,
\end{equation}
\begin{equation}\label{quattro}
    \sum_{j=1}^k\|f_j\|^2_{L^2}+\limsup_{n\to\infty}\|w^{(n)}_k\|^2_{L^2}\leq 1,
\end{equation}
\begin{equation}\label{cinque}
    \lim_{k\to\infty}\limsup_{n\to\infty}\|w^{(n)}_k\|_{M^\infty}=0.
\end{equation}
Moreover,
\begin{equation}\label{cinqueprima}
\pi\big(z^{(n)}_j\big)^* w^{(n)}_k \to 0 \mbox{ weakly in }L^2,
\mbox{ as }n\to\infty, \mbox{ for each }k \geq 1\mbox{ and }1\leq j \leq k,
\end{equation}
and
\begin{equation}\label{eq agg}
\lim_{n\to\infty} \lV \sum_{j=1}^k \pi(z^{(n)}_j)f_j \rV^2_{L^2}=\sum_{j=1}^k\|f_j\|^2_{L^2}.
\end{equation}
\begin{remark}\label{rem_a}
As a consequence of the decomposition, we see that $\pi\big(z^{(n)}_k\big)^*f^{(n)}$ converges weakly to $f_k$ in $L^2$, as $n\to\infty$, for each $k \geq 1$.
\end{remark}
For simplicity, we have introduced the time-frequency profile decomposition as an application of the abstract theory of dislocation sets \cite{tintarev_book}. Alternatively, it would have also been possible to derive the decomposition from the theory of \emph{dislocation groups} \cite[\red{Section} 4.5.2]{tao}, by considering the \emph{reduced Weyl-Heisenberg group} $\{\lambda \pi(z): z \in \rdd, |\lambda|=1\}$ and by exploiting the compactness of the set of phase factors $\lambda$.

\section{Proof of the main result (Theorem \ref{mainthm})}\label{sec proof mainthm}
In the following lemmas we deal with  $k$-tuples \red{$h=(h_1,\ldots,h_k)$} of functions on a measure space endowed with a $\sigma$-finite measure and we use the notation $\ell^r(L^p)$ for the corresponding vector-valued norm:
\begin{align}\label{eq_F}
\|h\|_{\ell^r(L^p)}=\lc \sum_{j=1}^k \|h_j\|_{L^p}^r \rc^{1/r}.
\end{align}

The following result provides a version of the classical Riesz-Thorin interpolation theorem for linear operators defined only on some finite dimensional subspace of simple functions. 
\begin{lemma}\label{lemma one}
Let $N$ be a positive integer and let $E_1,\ldots E_N$ be disjoint measurable sets of finite measure. Let 
\[
X=X_{E_1,\ldots,E_N}={\rm span}_{\C}\{\chi_{E_1}, \ldots, \chi_{E_N}\}.
\]

Let $1\leq p_0,q_0,p_1,q_1,r_0,r_1,p,q,r\leq\infty$ and $0<\theta<1$, with $1/p=(1-\theta)/p_0+\theta/p_1$, $1/q=(1-\theta)/q_0+\theta/q_1$, $1/r=(1-\theta)/r_0+\theta/r_1$, $p,r<\infty$.

Let $T:X^k\to L^{q_0}+L^{q_1}$ be a linear operator and $M_0,M_1>0$ such that
\begin{equation}\label{eq lemma3.1 uno}
\|T h\|_{L^{q_0}}\leq M_0 \|h\|_{\ell^{r_0}(L^{p_0})}
\end{equation}
and 
\begin{equation}\label{eq lemma3.1 due}
\|T h\|_{L^{q_1}}\leq M_1 \|h\|_{\ell^{r_1}(L^{p_1})}
\end{equation}
for every $h\in X^k$. Then 
\begin{equation}\label{eq lemma3.1 tre}
\|T h\|_{L^q}\leq M_0^{1-\theta} M_1^\theta \|h\|_{\ell^{r}(L^{p})}
\end{equation}
for every $h\in X^k$. 
\end{lemma}
\begin{proof}
A function $h=(h_1,\ldots,h_k) \in X^k$ can be written uniquely as
\begin{align*}
h_j = \sum_{i=1}^N a^i_j \,\chi_{E_i}, 
\end{align*}
with $a^i_j \in \mathbb{C}$. In terms of the coefficients, the norm \eqref{eq_F} reads
\begin{align*}
\|h\|_{\ell^r(L^p)} = \Big( \sum_{j=1}^k \Big( \sum_{i=1}^N |a^i_j|^p
|E_i| \Big)^{r/p} \Big)^{1/r}.
\end{align*}
Therefore, the conclusion follows from standard interpolation results with respect to weighted sequences spaces; see, e.g., \cite[Theorems 5.1.1, 5.1.2, and 5.6.3]{bl76}. Alternatively, one can see that the standard proof of the Riesz-Thorin theorem can be carried out within the space $X^k$.
\end{proof}
\begin{remark}
Lemma \ref{lemma one} fails if the sets $E_1,\ldots,E_n$ are not disjoint. In particular, the conclusion does not generalize to all finite dimensional subspaces $X$ of simple functions. For example, let $k=1$ (scalar case), the measure space $\{0,1,2\}$ be endowed with the counting measure, $E_1=\{0,1\}$, $E_2=\{1,2\}$ and $X={\rm span}_{\C}\{\chi_{E_1},\chi_{E_2}\}$, $Th=(h(0)+h(1)+h(2))\chi_{\{0\}}$ for $h\in X$, $1\leq q_0,q_1,r_0,r_1\leq\infty$ arbitrary, $p_0=1$, $p_1=\infty$, $p=1/\theta=2$. Then, \eqref{eq lemma3.1 uno} holds with $M_0=1$ and \eqref{eq lemma3.1 due} holds with $M_1=2$, while the function
$h=\chi_{E_1}+\chi_{E_2}$ provides a counterexample to the corresponding estimate \eqref{eq lemma3.1 tre}.
%$a,b\in\C$, \eqref{eq lemma3.1 uno} is equivalent to the inequality
%\[
%2|a+b|\leq M_0(|a|+|a+b|+|b|),
%\]
%which holds with $M_0=1$, and \eqref{eq lemma3.1 due} is equivalent to 
%\[
%2|a+b|\leq M_1 \max\{|a|,|a+b|,|b|\},
%\]
%which holds with $M_1=2$. On the other hand \eqref{eq lemma3.1 tre} is equivalent to 
%\[
%2|a+b|\leq \sqrt{2}(|a|^2+|a+b|^2+|b|^2)^{1/2},
%\]
%which is false - for $a=b=1$ one obtains $4\leq \sqrt{12}$. 
\end{remark}
As an application of Lemma \ref{lemma one} we obtain the following asymptotic interpolation estimate for sequences of operators.

\begin{lemma}\label{lemma two}
With the same notation of Lemma \ref{lemma one}, let  $T_n:\ell^{r_0}(L^{p_0})+\ell^{r_1}(L^{p_1})\to L^{q_0}+L^{q_1}$, $n\in\mathbb{N}$, be a sequence of linear operators. 

Suppose that $T_n$ is bounded $\ell^{r_0}(L^{p_0})\to L^{q_0}$ for every $n\in\mathbb{N}$,  with
\begin{equation}\label{eq lemma limsup}
\limsup_{n\to\infty}\|T_n h\|_{L^{q_0}}\leq M_0 \|h\|_{\ell^{r_0}(L^{p_0})},
\end{equation}
and, for every $n\in\mathbb{N}$,
\begin{equation}\label{eq lemma limsup3}
\|T_n h\|_{L^{q_1}}\leq M_1 \|h\|_{\ell^{r_1}(L^{p_1})}.
\end{equation}
Then 
\begin{equation}\label{eq lemma limsup2}
\limsup_{n\to\infty}\|T_n h\|_{L^q}\leq M_0^{1-\theta} M_1^\theta \|h\|_{\ell^{r}(L^{p})}.
\end{equation}
\end{lemma}
\begin{proof}
First, we prove that \eqref{eq lemma limsup2} holds when $h$ is a $k$-tuple of simple functions; hence $h$ belongs to some space $X^k$, with $X=X_{E_1,\ldots,E_N}$ and $E_1,\ldots,E_N$ of finite measure and  pairwise disjoint, as in Lemma \ref{lemma one}. 

The family of operators $T_n: \ell^{r_0}(L^{p_0})\to L^{q_0}$ is equicontinuous by the uniform boundedness principle, hence the estimate \eqref{eq lemma limsup} holds uniformly with respect to $h$ when $h$ belongs to a compact subset of $\ell^{r_0}(L^{p_0})$. Precisely, for every $\epsilon>0$ there exists $n_0\in\mathbb{N}$ such that   \[
\|T_n h\|_{L^{q_0}}\leq M_0 \|h\|_{\ell^{r_0}(L^{p_0})}+\epsilon
\]
for every $n\geq n_0$ and every $h$ in a compact subset of $\ell^{r_0}(L^{p_0})$. In particular, this holds for functions $h$ normalized in $\ell^{r_0}(L^{p_0})$ and in the finite dimensional space $X^k$, with $X=X_{E_1,\ldots E_N}$ as above. By homogeneity, we deduce that for every $\epsilon>0$ there exists $n_0\in\mathbb{N}$ such that
\[
\|T_n h\|_{L^{q_0}}\leq (M_0+\epsilon) \|h\|_{\ell^{r_0}(L^{p_0})}
\]
for $h\in X^k$, $n\geq n_0$. 
By \eqref{eq lemma limsup3} and Lemma \ref{lemma one} we obtain
\[
\|T_n h\|_{L^q}\leq (M_0+\epsilon)^{1-\theta} M_1^\theta \|h\|_{\ell^{r}(L^{p})}
\]
for $h\in X^k$, $n\geq n_0$,
which implies \eqref{eq lemma limsup2} for $h \in X^k$. 

Since $p$ is assumed to be finite, the set of simple functions is dense in $L^p$ and the family of operators $T_n:\ell^{r}(L^{p})\to L^q$ is equicontinuous (by the assumptions and complex interpolation), so that \eqref{eq lemma limsup2} holds for every $h\in \ell^{r}(L^{p})$.
\end{proof}
\begin{remark}
It is is easy to see that the conclusion of Lemma \ref{lemma two} still holds if \eqref{eq lemma limsup3} is replaced by the assumption that  $T_n$ is bounded $\ell^{r_1}(L^{p_1})\to L^{q_1}$ for every $n\in\mathbb{N}$,  with
\[
\limsup_{n\to\infty}\|T_n h\|_{L^{q_1}}\leq M_1 \|h\|_{\ell^{r_1}(L^{p_1})},
\]
but we will not need this fact.
\end{remark}

We are now ready to prove our main result. 
\begin{proof}[Proof of Theorem \ref{mainthm}]
\noindent {\bf Step 1.} \emph{Profile decomposition}.
	
Let $L$ be the supremum in \eqref{uno}. 
Since $\Omega$ has positive measure, we have $L>0$. (Indeed, it is easy to see that there exists $f \in L^2(\mathbb{R}^d)$ such that  \red{$A(f) \not = 0$ on $\Omega$}; see, for example, the proof of Proposition \ref{pro counter 2} below for details.)
\red{Moreover,} from the pointwise estimate $|Af|\leq \|f\|^2_{L^2}$ we obtain $L\leq |\Omega|^{1/p}$, hence $L$ is finite.

Let  $f^{(n)}$ be a maximizing sequence, that we can assume normalized without loss of generality, that is $\|f^{(n)}\|_{L^2}=1$. After passing to a subsequence, we apply the profile decomposition described in Section \ref{subsec conc comp}. We use the notation introduced there and we also set $F^{(n)}_k\coloneqq \sum_{j=1}^k \pi\big(z^{(n)}_j\big)f_j$, so that $f^{(n)}=F^{(n)}_k+ w^{(n)}_k$. Observe for future reference that by \eqref{quattro} and \eqref{eq agg} we have 
\begin{equation}\label{eq agg bis}
\|F^{(n)}_k\|^2_{L^2}+\|w^{(n)}_k\|^2_{L^2}\leq 2,\quad n\geq n_k,
\end{equation}
for some $n_k\in\bN$ depending on $k$. 

Using the sesquilinearity of the cross-ambiguity distribution, we can write
\begin{align}\label{sei}
    A(f^{(n)})= \sum_{j=1}^k A(\pi(z^{(n)}_j)f_j)&+\sum_{\substack{1\leq j,j'\leq k \\ j\not=j'}} A(\pi(z^{(n)}_j)f_j,\pi(z^{(n)}_{j'})f_{j'})\\
    &+A(F^{(n)}_k,w^{(n)}_k)+ A(w^{(n)}_k,F^{(n)}_k)+A(w^{(n)}_k).\nonumber
\end{align}
We now study the asymptotic behavior, as $n\to\infty$, of the $L^p(\Omega)$-norm of the terms on the right-hand side of \eqref{sei}.

\noindent {\bf Step 2.} \emph{Asymptotic decoupling}.
A simple computation gives  
\begin{equation}\label{eq cov amb}
A(\pi(z^{(n)}_j)f_j)=M_{J z^{(n)}_j}A(f_j)
\end{equation}
where 
$J=\begin{pmatrix} 0& I\\-I&0 \end{pmatrix}$ is the canonical symplectic matrix and $M$ stands for the modulation operator (here in $\rdd$).

\red{For fixed $k\geq 1$ and $n\in\bN$}, consider the operator
\begin{align*}
T^{(n)}_k: L^2(\Omega)^k \to L^2(\Omega),
\qquad
T^{(n)}_k(h_1, \ldots, h_k) = \sum_{j=1}^k M_{J z^{(n)}_j} h_j,
\end{align*}
\red{where the $h_j$'s are understood extended by zero on $\rdd\setminus\Omega$.}
By \eqref{tre} and the Riemann-Lebesgue lemma
\begin{align}\label{eq_a}
\begin{aligned}
\lim_{n\to\infty} \lV 
T^{(n)}_k(h_1, \ldots, h_k) \rV_{L^2(\Omega)}^2
&=\lim_{n\to\infty}
\sum_{j,j'=1}^k \langle h_j \overline{h_{j'}}, M_{J \left(z^{(n)}_j-z^{(n)}_{j'}\right)} \chi_\Omega \rangle_{L^2(\rdd)}
\\
&=\sum_{j=1}^k \|h_j\|^2_{L^2(\Omega)}.
\end{aligned}
\end{align}
In addition, for any $1 \leq p_1 \leq \infty$,
\begin{align}\label{eq_b}
\lV T^{(n)}_k(h_1, \ldots, h_k)\rV_{L^{p_1}(\Omega)}=
\lV\sum_{j=1}^k M_{J z^{(n)}_j}h_j\rV_{L^{p_1}(\Omega)}\leq \sum_{j=1}^k \|h_j\|_{L^{p_1}(\Omega)}.
\end{align}
We now use Lemma \ref{lemma two} to interpolate
between \eqref{eq_b} with $p_1=1,\infty$ and \eqref{eq_a}, and obtain
\begin{align*}
\limsup_{n\to\infty}\lV\sum_{j=1}^k M_{J z^{(n)}_j}h_j\rV_{L^p(\Omega)}\leq \lc \sum_{j=1}^k \|h_j\|^{p^\ast}_{L^p(\Omega)}\rc^{1/p^\ast},
\end{align*}
with \[p^\ast=\min\{p,p'\}.\] Applying this estimate to $h_j=A(f_j)$ and combining it with \eqref{eq cov amb} we obtain the following \textit{asymptotic decoupling estimate}:
\begin{equation}\label{eq decoupling}  \limsup_{n\to\infty}\lV\sum_{j=1}^k A(\pi(z^{(n)}_j)f_j)\rV_{L^p(\Omega)}\leq \lc \sum_{j=1}^k \|A(f_j)\|^{p^\ast}_{L^p(\Omega)}\rc^{1/p^\ast}.
\end{equation}

\noindent {\bf Step 3.} \emph{The error terms in \eqref{sei}}.

Proceeding with the analysis of the other terms in \eqref{sei}, we have
\begin{equation}\label{otto}
    |A(\pi(z^{(n)}_j)f_j,\pi(z^{(n)}_{j'})f_{j'})|=|A(f_j,\pi(z^{(n)}_{j'}-z^{(n)}_j)f_{j'})|.
\end{equation}
Since $A(f,g)\to 0$ at infinity if $f,g\in L^2(\rd)$, we see by \eqref{tre} and the dominated convergence theorem (the expression in \eqref{otto} is $\leq 1$ in $\rdd$), that 
\begin{equation}\label{nove}
\|A(\pi(z^{(n)}_j)f_j,\pi(z^{(n)}_{j'})f_{j'})\|_{L^p(\Omega)}\to 0 \quad \text{as } n\to\infty,
\qquad (j\not=j').
\end{equation}

Second, concerning the term $A(F^{(n)}_k,w^{(n)}_k)$ in \eqref{sei}, in view of the embedding \eqref{eq l2 embed} and the estimate \eqref{eq est l2minfty} we deduce that, for every compact $K\subset\rdd$,
\begin{align*}
\| A(F^{(n)}_k,w^{(n)}_k)\|_{L^2(K)}&\lesssim_K
\| A(F^{(n)}_k,w^{(n)}_k)\|_{\mathcal{W}(L^2,L^\infty)}\\
&\lesssim \|F^{(n)}_k\|_{L^2}\|w^{(n)}_k\|_{M^\infty},
\end{align*}
where the implied constant is independent of $k,n$. Using \eqref{eq agg bis} we see that, for $n\geq n_k$,
\[
\| A(F^{(n)}_k,w^{(n)}_k)\|_{L^2(K)}\lesssim_K  \|w^{(n)}_k\|_{M^\infty}
\]
and
\begin{equation}\label{dieci}
\| A(F^{(n)}_k,w^{(n)}_k)\|_{L^\infty(\rdd)}\leq \| F^{(n)}_k\|_{L^2} \|w^{(n)}_k\|_{L^2} \leq 1.
\end{equation}
Hence, by H\"older's inequality, we conclude that, for $n\geq n_k$,
\[
\| A(F^{(n)}_k,w^{(n)}_k)\|_{L^p(K)}\lesssim_K \|w^{(n)}_k\|_{M^\infty}^{\theta_p},
\]
where $\theta_p=1$ for $1\leq p\leq 2$, and $\theta_p=2/p$ if $2<p<\infty$. 
We claim that this  implies 
\begin{equation}\label{undici}
    \lim_{k\to\infty}\limsup_{n\to\infty}
    \| A(F^{(n)}_k,w^{(n)}_k)\|_{L^p(\Omega)}=0.
\end{equation}
Indeed, this is clear from \eqref{cinque} if $\Omega$ is compact. When $\Omega$ is merely a measurable set with finite measure we can choose a compact subset $K\subset\Omega$ with $|\Omega\setminus K|$ arbitrarily small and use again the uniform bound \eqref{dieci} on $\Omega\setminus K$. 

The same argument shows that
\begin{equation}
    \label{dodici}
    \lim_{k\to\infty}\limsup_{n\to\infty}
    \| A(w^{(n)}_k,F^{(n)}_k)\|_{L^p(\Omega)}=0,
\end{equation}
and 
\begin{equation}
    \label{tredici}
    \lim_{k\to\infty}\limsup_{n\to\infty}
    \| A(w^{(n)}_k)\|_{L^p(\Omega)}=0.
\end{equation}

\noindent {\bf Step 4.} \emph{Existence of optimizers}.

By the very definition of $f^{(n)}$, \eqref{sei}, the triangle inequality and \eqref{eq decoupling}, \eqref{nove}, \eqref{undici}, \eqref{dodici}, \eqref{tredici}, letting first $n\to\infty$ and then $k\to\infty$, we obtain
\[
L=\lim_{n\to\infty} \|A(f^{(n)})\|_{L^p(\Omega)}\leq \lc \sum_{j=1}^\infty \|A(f_j)\|^{p^\ast}_{L^p(\Omega)}\rc^{1/p^\ast}.
\]
On the other hand, by the very definition of $L$, we have 
\begin{equation}\label{quattordici}
\|A(f_j)\|_{L^p(\Omega)}\leq L \|f_j\|_{L^2}^2
\end{equation}
so that from \eqref{quattro} we obtain
\begin{align*}
L=\lim_{n\to\infty} \|A(f^{(n)})\|_{L^p(\Omega)}&\leq \lc \sum_{j=1}^\infty \|A(f_j)\|^{p^\ast}_{L^p(\Omega)}\rc^{1/p^\ast}\\
&\leq L\lc\sum_{j=1}^\infty \|f_j\|^{2p^\ast}_{L^2}\rc^{1/p^\ast}
\leq L \sum_{j=1}^\infty \|f_j\|^{2}_{L^2}\leq L.
\end{align*}
We then see that all these inequalities must be equalities. If $p>1$, so that $p^\ast>1$, this is possible only if $f_j=0$ except for one $j$, say $j=1$, and $\|f_1\|_{L^2}=1$. Hence $f_1$ is a maximizer.

Moreover, concerning the claim in the statement for $1<p<\infty$,
from \eqref{due} we have $f^{(n)}=\pi(z^{(n)}_1) f_1+w^{(n)}_1$. Since, by Remark \ref{rem_a}, $\pi\big(z^{(n)}_1\big)^* f^{(n)}$ converges weakly to $f_1$ in $L^2$, and $\|f^{(n)}\|_{L^2}=\|f_1\|_{L^2}=1$, we deduce that $\pi\big(z^{(n)}_1\big)^* f^{(n)}$ converges to $f_1$ in $L^2$. 
Finally, by passing to a subsequence, the cocycles in \eqref{eq_dos}
can be assumed to converge, $\sigma(z^{(n)},-z^{(n)}) \to c$,
and therefore $\pi\big(-z^{(n)}_1\big) f^{(n)}$ converges to the maximizer $c f_1$. 

If $p=1$, then $p^\ast=1$, and the same chain of equalities implies, together with \eqref{quattordici}, that equality holds in \eqref{quattordici} for all $j$. Since $f_j\not=0$ for at least one $j$, such a $f_j$ will be a maximizer for the problem \eqref{uno}. The proof is then  concluded. 
\end{proof}

\begin{remark} If one is only interested in the existence of a maximizer, even in the case $p>1$, the conclusion would follow as in the case $p=1$, that is, by applying the triangle inequality to $\sum_{j=1}^k A(\pi(z^{(n)}_j)f_j)$ - hence ignoring the oscillations and without using Lemma \ref{lemma two}. The more elaborated argument given above is rewarded with the stronger conclusion for $1<p<\infty$.
\end{remark}

\section{Proofs of Propositions \ref{pro counter 2} and \ref{pro counterexample}}\label{sec proof propositions}
\begin{proof}[Proof of Proposition \ref{pro counter 2}]
 It is clear that $\|A(f)\|_{L^\infty(\Omega)}\leq \|f\|_{L^2}^2$. On the other hand, consider a point $(x_0,\omega_0)\in\Omega$ of positive Lebesgue density for $\Omega$ (which exists since $|\Omega|>0$). Using the covariance property of $A(f)$ under symplectic transformations recalled in \eqref{covariance} and the transitivity of the linear symplectic group on $\rdd\setminus\{0\}$, we can suppose $\omega_0=0$. Let $f(x)=2^{d/4} e^{-\pi|x|^2}$ and $f_\lambda(x)=\lambda^{-d/2}f(x/\lambda)$, $\lambda>0$. Explicit computations show that $\|f_\lambda\|_{L^2}=1$ and 
\[
|A(f_\lambda)(x,\omega)|=|V_{f_\lambda} f_\lambda(x,\omega)|= e^{-\frac{\pi}{2\lambda^2}|x|^2-\frac{\pi \lambda^2}{2}|\omega|^2},
\]
see, e.g., \cite[Proposition 1.48 and Appendix A]{folland}.
Since $A(f_\lambda)$ is continuous at the point $(x_0,0)$, of positive Lebesgue density for $\Omega$,
\[
\|A(f_\lambda)\|_{L^\infty(\Omega)}\geq e^{-\frac{\pi}{2\lambda^2}|x_0|^2},
\]
which implies
\[
\liminf_{\lambda\to+\infty} \|A(f_\lambda)\|_{L^\infty(\Omega)}\geq 1.
\]
Concerning the existence and characterization of maximizers, we invoke the following \emph{radar correlation estimate}:
\begin{align}\label{eq_d}
|A(f)(x,\omega)|<A(f)(0,0)=1, \quad \mbox{if } \|f\|_{L^2}=1;
\end{align}
see, e.g., \cite[Lemma 4.2.1]{grochenig_book}.
Hence, if $|\Omega\cap B_r|>0$ for every $r>0$, it 
follows from the continuity of $A(f)$ that every $f\in L^2(\rd)\setminus\{0\}$ is a maximizer.

If instead there exists $r_0>0$ such that $|\Omega\cap B_{r_0}|=0$, then
\[
\|A(f)\|_{L^\infty(\Omega)}= \|A(f)\|_{L^\infty(\Omega\setminus B_{r_0})}\leq \sup_{(x,\omega)\in \overline{\Omega\setminus B_{r_0}}}|A(f)(x,\omega)|.
\]
Since $A(f)$ vanishes at infinity, in view of \eqref{eq_d}, this last supremum is still $<1$, because it is attained at some point of the closed set $\overline{\Omega\setminus B_{r_0}}$, which does not contain the origin. 
\end{proof}
\begin{proof}[Proof of Proposition \ref{pro counterexample}] Let us first prove \eqref{uno-bis}. From the trivial pointwise estimate $|\langle f,T_x f\rangle|\leq \|f\|_{L^2}^2$ it is clear that the supremum in \eqref{uno-bis} is $\leq |\Omega|^{1/p}$. On the other hand, for $\lambda>0$ let $B_\lambda$ be the open ball in $\rd$ with center $0$ and radius $\lambda$ and denote by $\chi_\lambda$ its characteristic function. For $K\subset\Omega$ compact, set $M_K=\max\{|x|: x\in K\}$. 

Then, for $x\in K$ and $\lambda\geq M_K$ we have \[
\langle \chi_\lambda, T_x \chi_\lambda\rangle\geq |B_{\lambda-M_K}|
\]
so that, for $\lambda\geq M_K$,
\begin{align*}
\frac{\lc\int_\Omega |\langle \chi_\lambda,T_x \chi_\lambda\rangle|^p dx\rc^{1/p}}{\|\chi_\lambda\|^2_{L^2}}&\geq \frac{\lc\int_K |\langle \chi_\lambda,T_x \chi_\lambda\rangle|^p dx\rc^{1/p}}{\|\chi_\lambda\|^2_{L^2}}\\
&\geq \frac{|B_{\lambda-M_K}| |K|^{1/p}}{|B_{\lambda}|},
\end{align*}
which implies 
\[
\liminf_{\lambda\to+\infty} \frac{\lc\int_\Omega |\langle \chi_\lambda,T_x \chi_\lambda\rangle|^p dx\rc^{1/p}}{\|\chi_\lambda\|^2_{L^2}} \geq |K|^{1/p}
\]
as $\lambda\to+\infty$. Since $|\Omega\setminus K|$ can be arbitrarily small, \eqref{uno-bis} is proved. 

Let us now prove that there is no extremal function. Suppose on the contrary that $f\in L^2(\rd)\setminus\{0\}$ is such an extremal function, which we can further assume to be normalized in $L^2$: $\|f\|_{L^2}=1$. Then  
\[
\int_\Omega |\langle f,T_x f\rangle|^p dx=|\Omega|,
\]
which together with the estimate $|\langle f,T_x f\rangle|\leq 1$ implies that 
\[
|\langle f,T_x f\rangle|=1
\]
for almost every $x\in \Omega$. Hence, since $|\Omega|>0$, there exists $x_0\in\rd$, $x_0\not=0$, $c\in\C$, $|c|=1$, such that 
\[
T_{x_0}f=c f.
\]
Taking the Fourier transform we \red{obtain that $f=0$, hence a contradiction.}
\end{proof}

\section{Optimization with fixed window}\label{sec fixed wind} 
To put our main result into context, we now mention the problem of the optimization of the cross-ambiguity when one of the arguments is kept fixed (or, equivalently, the optimization of the short-time Fourier transform with a fixed window). As we show below, the existence of optimizers is in this case much easier to prove --- while the characterization of such extremizers with, for example, the Gaussian window, is a challenging subject \cite{nicola_tilli}. 

\begin{proposition}\label{pro fixed wind} Let $g\in L^2(\rd) \setminus \{0\}$ and $\Omega\subset\rdd$ be a measurable subset of finite, positive measure. Let $1\leq p<\infty$. Then the supremum 
\begin{equation}\label{opt prob fixed wind}
    \sup_{f\in L^2(\rd)\setminus\{0\}}\frac{\lc\int_\Omega |A(f,g)(x,\omega)|^p dx d\omega\rc^{1/p}}{\|f\|_{L^2}}
\end{equation}
    is attained. Moreover, any maximizing sequence that is normalized in $L^2(\rd)$ has a subsequence that converges in $L^2$ to a maximizer.  
\end{proposition}
 
\begin{proof}[Proof of Proposition \ref{pro fixed wind}]
We claim that the functional $f\mapsto \|A(f,g)\|_{L^p(\Omega)}$ is sequentially  weakly continuous on $L^2(\rd)$. Indeed, if $f^{(n)}$ converges weakly to $f\in L^2(\rd)$, it follows at once from the definition of the cross-ambiguity function that $A(f^{(n)},g)\to A(f,g)$ pointwise in $\rdd$, and moreover $|A(f^{(n)},g)|\leq \|f^{(n)}\|_{L^2}\|g\|_{L^2}\lesssim 1$ on $\rdd$, so that the claim follows from the dominated convergence theorem. 

\red{Let $f^{(n)}$ be a maximizing sequence with 
$\|f^{(n)}\|_{L^2}=1$, and let $L$ be the supremum in \eqref{opt prob fixed wind}. Since $|\Omega|>0$ and $g \not\equiv 0$, it follows that $L>0$. Indeed, it is sufficient to consider a point $z_0\in\Omega$ of positive Lebesgue density for $\Omega$ and observe that the function $A(\pi(z_0)g,g)$ is continuous and $|A(\pi(z_0)g,g)(z_0)|=\|g\|_{L^2}^2>0$. 

 Then $f^{(n)}$  has a subsequence, that we still denote by $f^{(n)}$, weakly convergent to some $f\in L^2(\rd)$, and, by the above mentioned sequential weak continuity, \linebreak $\|A(f,g)\|_{L^p(\Omega)}=L$. Since $L>0$, $f\not=0$. 
\red{In addition, $\|f\|_{L^2}\leq \liminf_{n\to\infty}\|f^{(n)}\|_{L^2}\leq 1$,
 so that $f$ is a maximizer and $\|f\|_{L^2}=1$.} As a consequence, $f^{(n)}\to f$ in $L^2$. }
\end{proof}
\begin{remark}\label{rem localiz}
For $p=2$, the existence of a maximizer for the problem \eqref{opt prob fixed wind} also follows from the spectral properties of the non-negative bounded operator $V_g^\ast \chi_\Omega V_g$ on $L^2(\rd)$. Indeed, $|A(f,g)|=|V_g f|$, so that 
\[
\int_\Omega |A(f,g)(x,\omega)|^2 \, dxd\omega =\langle V_g^\ast \chi_\Omega V_g f,f\rangle. 
\]
Since $|\Omega|<\infty$, the operator $ V_g^\ast \chi_\Omega V_g$ is compact (in fact, trace class \cite{cordero_grochenig,wong_book}), so that any eigenfunction corresponding to the maximum eigenvalue is a maximizer for the problem \eqref{opt prob fixed wind} (with $p=2$).

\end{remark}
We emphasize that for the optimization problem in Theorem \ref{mainthm} we could not have argued as in the proof of Proposition \ref{pro fixed wind}, because of the lack of sequential weak upper semicontinuity of the corresponding functional, as shown below.
\begin{proposition}\label{pro counter}
Let  $\Omega\subset\rdd$ be a measurable subset of finite, positive measure. The functional $\|A(f) \|_{L^2(\Omega)}$ on $L^2(\rd)$ is not sequentially weakly upper semicontinuous at any point. 
\end{proposition}
\begin{proof}
The computations in the proof of Proposition \ref{pro counter 2} show that there exists $g\in L^2(\rd)$ - in fact, a Gaussian function - such that $\|A(g)\|_{L^2(\Omega)}>0$, since $|\Omega|>0$. 

Let now $f\in L^2(\rd)$. Then $f+\pi(z)g$ converges weakly to $f$ as $|z|\to+\infty$. On the other hand, 
\[
\|A(f+\pi(z)g)\|^2_{L^2(\Omega)}=\|A(f)+A(f,\pi(z)g)+A(\pi(z)g,f)+M_{Jz} A(g)\|^2_{L^2(\Omega)},
\]
where we used \eqref{eq cov amb}. 

Since $A(f,\pi(z)g)+A(\pi(z)g,f)\to 0$ in $L^2(\Omega)$ as $|z|\to+\infty$, by arguing as in \eqref{eq_a} we obtain 
\[
\lim_{|z|\to+\infty} \|A(f+\pi(z)g)\|^2_{L^2(\Omega)}=\|A(f)\|^2_{L^2(\Omega)}+\|A(g)\|^2_{L^2(\Omega)}>\|A(f)\|^2_{L^2(\Omega)},
\]
which gives the desired conclusion. 
\end{proof}

\section{Variations on the main result}\label{sec_var}
\subsection{The optimization problem in modulation spaces}\label{sec mod}

We now derive a variant of Theorem \ref{mainthm}, where the function is optimized over the modulation space $M^q(\rd)$, $0< q< 2$. For the precise formulation, fix a window function $g\in\mathcal{S}(\rd)\setminus\{0\}$ and $0< q\leq \infty$; then $M^q(\rd)$ is defined as the space of temperate distributions $f\in \mathcal{S}'(\rd)$ such that
\[
\|f\|_{M^q}\coloneqq \|V_g f\|_{L^q(\rdd)}<\infty. 
\]
Different windows $g$ give rise to the same space with equivalent norms. Moreover, $M^2(\rd)=L^2(\rd)$ with equivalent norms, and $M^{q_1}(\rd)\hookrightarrow M^{q_2}(\rd)$ if and only if $0<q_1\leq q_2\leq\infty$; see \cite{benyimodulation, galperin} and \cite[Chapter 10]{grochenig_book} for background.

Thus, a modulation-space norm estimate $\|f\|_{M^q} \leq 1$ prescribes a certain  \red{integrability and decay} for a function $f$. The next result allows one to incorporate such constraints into the optimization of the ambiguity function.
\begin{theorem}\label{mainthm2}
Let $\Omega\subset\rdd$ be a measurable subset of finite, positive measure, and $1\leq p<\infty$, $0< q< 2$. Then the supremum 
\begin{equation}\label{uno-ter}
    \sup_{f\in M^q(\rd)\setminus\{0\}}\frac{\lc\int_\Omega |A(f)(x,\omega)|^p dx d\omega\rc^{1/p}}{\|f\|^2_{M^q}}
\end{equation}
    is attained. Moreover if $f^{(n)}$ is any maximizing sequence normalized in $M^q(\rd)$, then there exists a subsequence (still denoted by $f^{(n)}$) and $z^{(n)}\in\rdd$ such that $\pi(-z^{(n)})f^{(n)}$ converges in $M^q$ to a maximizer.
\end{theorem}
\begin{proof}[Proof of Theorem \ref{mainthm2}]
\noindent {\bf Step 1}. \emph{Profile decomposition in $M^q$}.

	The first part of the proof is
similar to (in fact, simpler than) the one of Theorem \ref{mainthm} and it will only be sketched. Let $L$ be the supremum in \eqref{uno-ter}; as in the proof of Theorem \ref{mainthm}, we note that $L>0$ since $|\Omega|>0$.

 Consider a maximizing sequence $f^{(n)}$, now normalized in $M^q$:  $\|f^{(n)}\|_{M^q}=1$. Since $0< q< 2$, we have $M^q\hookrightarrow L^2$, so that the sequence $f^{(n)}$ is bounded in $L^2$ and we can apply (after passing to a suitable subsequence) the profile decomposition in $L^2$ as described in Section \ref{subsec conc comp}, albeit with minor modifications; cf. \cite[Theorem 4.5.3]{tao} or \cite[Theorem 3.1 and its proof]{tintarev_book}. The formulas  \eqref{due}, \eqref{tre}, \eqref{cinque} and \eqref{eq agg} hold, whereas \eqref{quattro} is now replaced by 
\[
\sum_{j=1}^k\|f_j\|^2_{L^2}+\limsup_{n\to\infty}\|w^{(n)}_k\|^2_{L^2}\leq C
\]
\red{for some $C>0$}, since $\limsup_{n\to\infty}\|f^{(n)}\|_{L^2}$ is still finite but no longer necessarily $\leq 1$.
While this is sufficient to prove \eqref{nove}, \eqref{undici}, \eqref{dodici}, \eqref{tredici}, Step 4 of the proof of Theorem \ref{mainthm} requires some modifications. To complete the proof, we will prove that the profiles $f_j$ are not merely in $L^2(\rd)$ but actually belong to $M^q(\rd)$, and, moreover, satisfy the following precise norm estimate:
\begin{equation}\label{eq decomp mq}
\sum_{j=1}^\infty \|f_j\|_{M^q}^q\leq 1.
\end{equation}
Postponing the proof of this fact, let us see how to deduce the existence of optimizers. We start from the expansion \eqref{sei} for $A(f^{(n)})$. By the triangle inequality and \eqref{eq cov amb}, \eqref{nove}, \eqref{undici}, \eqref{dodici}, \eqref{tredici} we obtain
\[
L=\lim_{n\to\infty}\|A(f^{(n)})\|_{L^p(\Omega)}\leq \sum_{j=1}^\infty \|A(f_j)\|_{L^p(\Omega)}.
\]
By the definition of $L$,
\[
    \|A(f_j)\|_{L^p(\Omega)}\leq L\|f_j\|_{M^q}^2,
\]
and, since $q<2$,
\begin{align*}
L=\lim_{n\to\infty} \|A(f^{(n)})\|_{L^p(\Omega)}\leq \sum_{j=1}^\infty \|A(f_j)\|_{L^p(\Omega)}&\leq L\sum_{j=1}^\infty \|f_j\|_{M^q}^2\\
&\leq L\lc\sum_{j=1}^\infty \|f_j\|_{M^q}^q\rc^{2/q}\leq L.
\end{align*}
This implies that all $f_j$ are zero except one, say $f_1$, and $\|f_1\|_{M^q}=1$. Hence $f_1$ is a maximizer.

Finally, since, by Remark \ref{rem_a}, $\pi\big(z^{(n)}_1\big)^* f^{(n)}$ converges weakly (in $L^2$) to $f_1$, it turns out that  $V_g(\pi\big(z^{(n)}_1\big)^* f^{(n)})\to V_g f_1$ pointwise. Moreover $\|V_g(\pi\big(z^{(n)}_1\big)^* f^{(n)})\|_{L^q}= 
\|f^{(n)}\|_{M^q}=1=\|f_1\|_{M^q}=\|V_g f_1\|_{L^q}$, so that $V_g(\pi\big(z^{(n)}_1\big)^* f^{(n)})$ tends to $V_g f_1$ in $L^q$ by the Br\'ezis-Lieb Lemma \cite{MR699419, lieb_book}, i.e.  $\pi\big(z^{(n)}_1\big)^* f^{(n)}\to f_1$ in $M^q$. We now invoke \eqref{eq_dos} and eliminate the cocycles as in the proof of Theorem \ref{mainthm}.

\noindent {\bf Step 2}. \emph{Precise norm estimate for the profiles}.

We now prove \eqref{eq decomp mq}. As noted in Remark \ref{rem_a}, each $f_j$ is indeed the weak limit (in $L^2$) of (adjoint) time-frequency shifts of $f^{(n)}$, which are assumed to be normalized in $M^q(\rd)$. Since $V_g f_j$ is then the pointwise limit of the corresponding short-time Fourier transforms, we see that $f_j\in M^q(\rd)$ by Fatou's lemma. 

Moreover, \eqref{cinqueprima} implies that $V_g(\pi(-z^{(n)}_j)w^{(n)}_k)(z)=V_g(w^{(n)}_k)(z+z^{(n)}_j)$ tends to zero uniformly on compact subsets of $\rdd$ as $n\to\infty$ --- due to the strong continuity of time-frequency shifts. 

Suppose first that $1\leq q<2$. \red{For fixed $k\geq 1$,
given} $\epsilon>0$ there exist therefore compact subsets $K_j\subset\rdd$, $j=1,\ldots, k$, and $n_k\in\bN$ such that ($L^q$ standing for $L^q(\rdd)$)
\[
\|V_g(\pi(z^{(n)}_j) f_j)\chi_{\rdd\setminus(z^{(n)}_j+K_j)}\|_{L^q} =\|V_g (f_j)\chi_{\rdd\setminus K_j}\|_{L^q}<\epsilon
\]
and 
\[
\sum_{j=1}^k \| V_g(w^{(n)}_k)\chi_{z^{(n)}_j+K_j} \|_{L^q}<\epsilon
\]
for $n\geq n_k$. For each such $n$, by \eqref{due} and the triangle inequality,
\begin{align*}
1&=\| V_g f^{(n)}\|_{L^q}\\
&\geq \lV\sum_{j=1}^k V_g(\pi(z^{(n)}_j) f_j) \chi_{z^{(n)}_j+K_j}+V_g(w^{(n)}_k) \chi_{\rdd\setminus\cup_{j=1}^k (z^{(n)}_j+K_j)}\rV_{L^q}-(k+1)\epsilon.
\end{align*}
On the other hand, by \eqref{tre}, if $n$ is large enough the compact subsets $z^{(n)}_j+K_j$, $j=1,\ldots,k$, are pairwise disjoint (in the last summation we can consider just the indices $j$ such that $f_j\not=0$), so that 
\begin{align*}
\lV\sum_{j=1}^k V_g(\pi(z^{(n)}_j) f_j) \chi_{z^{(n)}_j+K_j}+&V_g(w^{(n)}_k) \chi_{\rdd\setminus\cup_{j=1}^k (z^{(n)}_j+K_j)}\rV_{L^q}^q\\
&\geq \sum_{j=1}^k\| V_g(\pi(z^{(n)}_j) f_j) \chi_{z^{(n)}_j+K_j}\|^q_{L^q}\\
&= \sum_{j=1}^k\| V_g (f_j) \chi_{K_j}\|^q_{L^q}\\
&\geq \sum_{j=1}^k (\|f_j\|_{M^q}-\epsilon)_+^q,
\end{align*}
where $(\cdot )_+$ denotes the positive part function.

In conclusion we have 
\[
\sum_{j=1}^k (\|f_j\|_{M^q}-\epsilon)_+^q\leq (1+(k+1)\epsilon)^q.
\]
Since $\epsilon$ and $k$ are arbitrary, we have proved \eqref{eq decomp mq} in the case $1 \leq q < 2$.

The argument needs to be slightly adapted for $0<q<1$. In this case we choose the compact subsets $K_j\subset\rdd$, $j=1,\ldots,k$, and $n_k\in\bN$ so that 
\[
\|V_g(\pi(z^{(n)}_j) f_j)\chi_{\rdd\setminus(z^{(n)}_j+K_j)}\|_{L^q}^q =\|V_g (f_j)\chi_{\rdd\setminus K_j}\|_{L^q}^q<\epsilon
\]
and 
\[
\sum_{j=1}^k \| V_g(w^{(n)}_k)\chi_{z^{(n)}_j+K_j} \|_{L^q}^q<\epsilon
\]
for $n\geq n_k$. Again, by \eqref{due} and the triangle inequality, now for $\|\cdot\|_{L^q}^q$,
\begin{align*}
1&=\| V_g f^{(n)}\|^q_{L^q}\\
&\geq \lV\sum_{j=1}^k V_g(\pi(z^{(n)}_j) f_j) \chi_{z^{(n)}_j+K_j}+V_g(w^{(n)}_k) \chi_{\rdd\setminus\cup_{j=1}^k (z^{(n)}_j+K_j)}\rV_{L^q}^q-(k+1)\epsilon.
\end{align*}
An argument similar to that used in the previous case now gives
\[
\red{\sum_{j=1}^k (\|f_j\|^q_{M^q}-\epsilon)\leq 1+(k+1)\epsilon,}
\]
which implies
\eqref{eq decomp mq} also for $0<q<1$.
\end{proof}

\subsection{Optimization with respect to Gabor systems}

While the constraint $f \in M^q(\mathbb{R}^d)$ in Theorem \ref{mainthm2} is independent of the choice of the window $g \in \mathcal{S}(\mathbb{R}^d)$, the functional optimized in \eqref{uno-ter} does depend on $g$ because it involves the window-dependent (quasi-)norm $\|f\|_{M^q}$. In practice, such norms are often replaced by certain discrete counterparts computed in terms of so-called \emph{Gabor systems}.

Precisely, consider a full-rank lattice $\Lambda\subset\rdd$ and $g\in\cS(\rd)$ such that the set of functions $\{\pi(\lambda)g\}_{\lambda\in\Lambda}$ is a \emph{frame for $L^2(\rd)$}, i.e.,
\[
\|f\|^2_{L^2}\lesssim \sum_{\lambda\in\Lambda}|\langle f, \pi(\lambda) g\rangle|^2\lesssim \|f\|^2_{L^2}.
\]
Then it turns out that the quantity
\begin{equation}\label{eq_aa}
|f|_{M^q} \coloneqq \lc\sum_{\lambda\in\Lambda} |\langle f, \pi(\lambda) g\rangle|^q\rc^{1/q}
\end{equation}
(with obvious changes if $q=\infty$) gives an equivalent (quasi-)norm in $M^q(\rd)$, $0<q\leq\infty$ \cite{benyimodulation,galperin}, \cite[Chapter 10]{grochenig_book}. The next result is an analog of Theorem \ref{mainthm2} for the discrete (quasi-)norm \eqref{eq_aa}.

\begin{theorem}\label{mainthm3}
The statement in Theorem \ref{mainthm2} is still valid if \eqref{uno-ter} is replaced by
\begin{equation*}
\sup_{f\in M^q(\rd)\setminus\{0\}}\frac{\lc\int_\Omega |A(f)(x,\omega)|^p dx d\omega\rc^{1/p}}{|f|_{M^q}^2}.
\end{equation*}
\end{theorem}
The derivation of Theorem \ref{mainthm3} requires minimal adaptations. Indeed, the map $\Lambda\ni\lambda\to\pi(\lambda)$ is still a projective unitary representation on $L^2(\rd)$, and
the corresponding operators $\{\pi(\lambda)\}_{\lambda\in\Lambda}$ still define a dislocation set. The corresponding notion of $D$-weak convergence reads
\[
\sup_{\lambda\in{\Lambda}} |\langle f_n-f,\pi(\lambda) h\rangle|\to 0
\]
for every $h\in L^2(\rd)$, and still implies convergence in $M^\infty$, due to the equivalence of the $|\cdot|_{M^\infty}$ and $\|\cdot\|_{M^\infty}$ norms. Thus, profile decompositions as in Section \ref{subsec conc comp} exist, now with $z^{(n)}_j\in\Lambda$. The proof of Theorem \ref{mainthm2} adapts almost verbatim ---  even in notation, by replacing the $L^q$ (quasi-)norm in $\rdd$ with respect to the Lebesgue measure by the $L^q$ (quasi-)pseudo-norm in $\rdd$ with respect to the Radon measure $\sum_{\lambda\in\Lambda}\delta_\lambda$.  The key point is that such a measure is invariant under the translations $z\mapsto z+ z^{(n)}_j$, because $z^{(n)}_j\in\Lambda$.

\section*{Acknowledgments}

The authors are very grateful to Karlheinz Gr\"ochenig for bringing to their attention the problem solved here, in connection to an unpublished manuscript of his and Markus Neuhauser.

The present research has been partially supported by the MIUR grant Dipartimenti di Eccellenza 2018-2022, CUP: E11G18000350001, DISMA, Politecnico di Torino. J. L. R. gratefully acknowledges support from the Austrian Science Fund (FWF): Y 1199.

S. I. T. is member of the Machine Learning Genoa (MaLGa) Center, Universit\`a di Genova. F. N. and S. I. T. are members of the Gruppo Nazionale per l’Analisi Matematica, la
Probabilit\`a e le loro Applicazioni (GNAMPA) of the Istituto Nazionale di Alta Matematica (INdAM).

 \section*{Statements and Declarations}

\noindent The authors declare no competing interests. Data sharing not applicable to this article as no datasets were generated or analysed during the current study.

%\bibliographystyle{abbrv}
%\bibliography{biblio.bib}
%\printbibliography
%\printbibitembibliography

\end{document}